\documentclass[11pt,reqno,fleqn]{article}
\usepackage{amsfonts,amssymb,amsmath,amsthm}
\usepackage[top=23mm,bottom=23mm,left=20mm,right=20mm]{geometry}
\usepackage[T1]{fontenc}
\usepackage{times}
\usepackage{multicol}
\usepackage{amsfonts}
\usepackage{graphicx}
\usepackage{graphics}
\usepackage{enumitem}
\usepackage{relsize}
\usepackage{cite}
\usepackage{fancyhdr}
\usepackage{color}
\fancypagestyle{titlepage}
{
   \fancyhf{}
   \lhead{The final publication is available at IOS Press through http://dx.doi.org/10.3233/JIFS-162050}
   \rhead{}
   \fancyfoot[C]{\thepage}
}

\newcommand\blfootnote[1]{%
  \begingroup
  \renewcommand\thefootnote{}\footnote{#1}%
  \addtocounter{footnote}{-1}%
  \endgroup
}
\newtheorem{thm}{Theorem}[section]
\numberwithin{equation}{section}
\newtheorem{defin}[thm]{Definition}

\newtheorem{cor}[thm]{Corollary} 
\newtheorem{lem}[thm]{Lemma} 
\newtheorem{pr}[thm]{Proposition}
\newtheorem{rem}[thm]{Remark}
\theoremstyle{definition}
\newtheorem{ex}[thm]{Example}

\usepackage{etoolbox}
\makeatletter
\patchcmd{\@maketitle}{\begin{center}}{\begin{flushleft}}{}{}
\patchcmd{\@maketitle}{\begin{tabular}[t]{c}}{\begin{tabular}[t]{@{}l}}{}{}
\patchcmd{\@maketitle}{\end{center}}{\end{flushleft}}{}{}
\makeatother
\begin{document}
\title{Comparison theorems for summability methods of sequences of fuzzy numbers}
\author{Enes Yavuz}
\date{{\small Department of Mathematics, Manisa Celal Bayar University, Manisa, Turkey.\\ E-mail: enes.yavuz@cbu.edu.tr}}
\maketitle
\thispagestyle{titlepage}
\blfootnote{\rule{0.63cm}{0cm}\emph{\!\!Mathematics Subject Classification:} 03E72, 40G10, 40D25}

\noindent\textbf{Abstract:}
\noindent In this study we compare Ces\`{a}ro and Euler weighted mean methods of summability of sequences of fuzzy numbers with Abel and Borel power series methods of summability of sequences of fuzzy numbers. Also, some results dealing with series of fuzzy numbers are obtained.

\section{Introduction}\allowdisplaybreaks
It is well known that the facts that human being met in the natural world are generally complex and inexact. Complexity and inexactness of real-world events often stems from uncertain nature of the parameters and from vague status of the underlying objects. Realizing that uncertainty is ubiquitous and essential in complex systems, researchers designed many uncertainty theories such as probability theory, evidence theory, fuzzy set theory to cope with problems of vagueness. Considered as the recent one, fuzzy set theory was introduced by Zadeh\cite{zadeh} in 1965 and since then theory has advanced in many branches of science and engineering. In mathematics, different classes of fuzzy numbers are introduced and various properties of these classes are investigated\cite{add1,add2,add3,add4}. In particular, classes of sequences of fuzzy numbers  are presented and convergence properties of sequences and series of fuzzy numbers are studied\cite{matloka,nanda,different1,different2,convergence1,convergence2,different3}. Besides, with the purpose of handling divergent sequences, summability methods of sequences of fuzzy numbers are defined and Tauberian conditions which guarantee the convergence of summable sequences are given\cite{summability0,summability1,sefapower,subrahmanyan,summability4}. Among them, Ces\`{a}ro, Euler weighted mean methods of summability and Abel, Borel power series methods of summability for sequences of fuzzy numbers have been studied recently and corresponding Tauberian theorems have been proved\cite{cesaro2,cesaro1,cesaro3,cesaro4,yavuzabel,yavuzborel,yavuzeuler}.

The main goal of this paper is to compare Ces\`{a}ro and Euler summability methods of sequences of fuzzy numbers with Abel and Borel summability methods, respectively. To achive this goal, in Section 3  we give an optimal bound for Ces\`{a}ro summable sequences of fuzzy numbers and prove a comparison theorem between Ces\`{a}ro and Abel methods of summability of sequences of fuzzy numbers. A Mertens' type result concerning multiplication of series of fuzzy numbers is also obtained.  In section 4 firstly we show that Euler summability method $E_p$ becomes stronger in summing up divergent sequences of fuzzy numbers as the order $p$ increases and then prove that $E_p$ convergence of a sequence of fuzzy numbers implies Borel convergence. Finally in Section 5, as results of comparisons made in Section 3-4, some Tauberian theorems for Abel and Borel methods of summability of sequences of fuzzy numbers have been extended to Ces\`{a}ro and Euler summability methods.
\section{Preliminaries}
A \textit{fuzzy number} is a fuzzy set on the real axis, i.e. u is normal, fuzzy convex, upper semi-continuous and $\operatorname{supp}u =\overline{\{t\in\mathbb{R}:u(t)>0\}}$ is compact \cite{zadeh}.
We denote the space of fuzzy numbers by $E^1$. \textit{$\alpha$-level set} $[u]_\alpha$ of $u\in E^1$
is defined by
\begin{eqnarray*}
[u]_\alpha:=\left\{\begin{array}{ccc}
\{t\in\mathbb{R}:u(t)\geq\alpha\} & , & \qquad if \quad 0<\alpha\leq 1, \\[6 pt]
\overline{\{t\in\mathbb{R}:u(t)>\alpha\}} & , &
if \quad \alpha=0.\end{array} \right.
\end{eqnarray*}
Let $u,v\in E^1$ and $k\in\mathbb{R}$. The addition and scalar multiplication are defined by
\begin{eqnarray*}
[u+v]_\alpha=[u^-_{\alpha}+v^-_{\alpha}, u^+_{\alpha}+v^+_{\alpha}], [k u]_\alpha=k[u]_\alpha
\end{eqnarray*}
where $[u]_\alpha=[u^-_{\alpha}, u^+_{\alpha}]$, for all $\alpha\in[0,1]$.

\begin{lem}\label{lemma}\cite{bede} The following statements hold:
\begin{itemize}
\item[(i)] $\overline{0} \in E^{1}$ is neutral element with respect to $+$, i.e., $u+\overline{0}=\overline{0}+u=u$ for all $u \in E^{1}$.
\item [(ii)] With respect to $\overline{0}$, none of $u \neq \overline{r}$, $r\in \mathbb{R}$
has opposite in $E^{1}.$
\item[(iii)] For any $a,b \in\mathbb{R}$ with $a, b \geq 0$ or $a,b \leq 0$ and any $u\in E^{1}$, we have
$(a + b) u = au + bu$. For general $a, b \in \mathbb{R}$, the above property does not hold.
\item[(iv)] For any $a \in \mathbb{R}$ and any $u, v \in E^{1}$, we have
$a (u+ v) = au + av.$
\item[(v)] For any $a, b \in \mathbb{R}$ and any $u \in E^{1}$, we have
$a (b u) = (a b) u.$
\end{itemize}
\end{lem}
The metric $D$ on $E^1$ is defined as
\begin{eqnarray*}
 D(u,v):=
\sup_{\alpha\in[0,1]}\max\{|u^-_{\alpha}-v^-_{\alpha}|,|u^+_{\alpha}-
v^+_{\alpha}|\}.
\end{eqnarray*}
\begin{pr}\label{pro}\cite{bede}
\label{p02} Let $u,v,w,z\in E^1$ and $k\in\mathbb{R}$. Then,
\begin{itemize}
\item [(i)] $(E^1,D)$ is a complete metric space.
\item [(ii)] $D(ku,kv)=|k|D(u,v)$.
\item [(iii)] $D(u+v,w+v)=D(u,w)$.
\item [(iv)] $D(u+v,w+z)\leq D(u,w)+D(v,z)$.
\item [(v)] $|D(u,\overline{0})-D(v,\overline{0})|\leq D(u,v)\leq
D(u,\overline{0})+D(v,\overline{0})$.
\end{itemize}
\end{pr}
A sequence $(u_n)$ of fuzzy numbers is said to be convergent to $\mu\in
E^1$ if for every $\varepsilon>0$ there exists an
$n_0=n_0(\varepsilon)\in\mathbb{N}$ such that $D(u_n,\mu)<\varepsilon~~\text{for all} ~~n\geq n_0.$ We mean that sequence $(u_n)$ converges to $\mu$ by $u_n\to\mu$.
\begin{defin}\label{series}\cite{kim}
Let $(u_k)$ be a sequence of fuzzy numbers. Then the expression $\sum u_k$ is called a
\textit{series of fuzzy numbers}. Denote $s_n=\sum_{k=0}^nu_k$ for
all $n\in\mathbb{N}$. If the sequence $(s_n)$ converges to a fuzzy
number $u$, then we say that the series $\sum u_k$ of fuzzy numbers
converges to $u$ and write $\sum u_k=u$ which implies as
$n\to\infty$ that
\begin{eqnarray*}
\sum_{k=0}^{n}u^{-}_{k}(\alpha)\rightarrow u^{-}(\alpha) ~\text{
and }~ \sum_{k=0}^{n}u^{+}_{k}(\alpha)\rightarrow u^{+}(\alpha)
\end{eqnarray*}
uniformly in $\alpha\in [0,1]$. Conversely, if the series
$\sum_{k}u^{-}_{k}(\alpha)=u^{-}(\alpha)$ and
$\sum_{k}u^{+}_{k}(\alpha)=u^{+}(\alpha)$ converge uniformly in
$\alpha$, then $u=\{(u^{-}(\alpha),u^{+}(\alpha)): \alpha\in
[0,1]\}$ defines a fuzzy number such that $u=\sum u_{k}$. We say otherwise the series of fuzzy numbers diverges.
\end{defin}
\begin{rem}\cite{yavuzeuler}\label{change}
Let $(u_n)$ be a sequence of fuzzy numbers. If $(x_n)$ is a sequence of non-negative real numbers, then
\begin{eqnarray*}
\sum\limits_{k=0}^nx_k\sum\limits_{m=0}^ku_m=\sum\limits_{m=0}^nu_m\sum\limits_{k=m}^nx_k
\end{eqnarray*}
holds by $(iii)$ and $(iv)$ of Lemma \ref{lemma}.
\end{rem}
\begin{thm}\label{triangle}\cite{tb}
If $\sum u_n$ and $\sum v_n$ converge, then $D\left(\sum u_n, \sum v_n\right)\leq \sum D(u_n, v_n).$
\end{thm}
\begin{thm}\label{absolute}\cite{tb}
If $\sum D(u_k, \bar{0}) <\infty$, then series $\sum u_k$ is convergent.
\end{thm}
Ces\`{a}ro, Euler weighted mean methods of summability  and Abel, Borel power series methods of summability for sequences of fuzzy numbers have been defined recently as the following:
\begin{defin}\cite{subrahmanyan}
Let $(u_n)$ be a sequence of fuzzy numbers and let sequence of arithmetic means of $(u_n)$ be defined by $\sigma_n=\frac{1}{n+1}\sum_{k=0}^{n}u_k$. We say that sequence $(u_n)$ is  Ces\`{a}ro summable to fuzzy number a $\mu$ if $\lim\limits_{n\to\infty}\sigma_n=\mu$.
\end{defin}
\begin{defin}\cite{yavuzeuler}
Let $(u_n)$ be a sequence of fuzzy numbers. The Euler means of $(u_n)$ is defined by
\begin{eqnarray*}
t^{p}_n=\frac{1}{(p+1)^n}\sum_{k=0}^n\binom{n}{k}p^{n-k}u_k\qquad\qquad (p>0).
\end{eqnarray*}
We say that $(u_n)$ is $E_p$ summable to a fuzzy number $\mu$ if $\lim\limits_{n\to\infty}t^{p}_n=\mu$.
\end{defin}
\begin{defin}\cite{yavuzabel}
A sequence $(u_n)$ of fuzzy numbers is said to be Abel summable to a fuzzy number $\mu$ if the series $\sum_{n=0}^{\infty}u_nx^n$ converges for all $x\in(0,1)$ and
\begin{eqnarray*}
\lim_{x\to 1^-}(1-x)\sum_{n=0}^{\infty}u_nx^n=\mu.
\end{eqnarray*}
\end{defin}
\begin{defin}\cite{yavuzborel}
A sequence $(u_n)$ of fuzzy numbers is said to be Borel summable to $\mu$ if the series $\sum_{n=0}^{\infty}\frac{x^n}{n!}u_n$ converges for all $x\in(0,\infty)$ and
\begin{eqnarray*}
\lim_{x\to\infty} e^{-x}\sum_{n=0}^{\infty}\frac{x^n}{n!}u_n=\mu.
\end{eqnarray*}
\end{defin}
\section{Comparison between Ces\`{a}ro and Abel methods of summability of sequences of fuzzy numbers}
In the following theorem we give an optimal bound for Ces\`{a}ro summable sequences of fuzzy numbers.
\begin{thm}\label{bound}
If sequence $(u_n)$ of fuzzy numbers is Ces\`{a}ro summable, then $D(u_n,\bar{0})=o(n)$ and this estimate is best possible.
\end{thm}
\begin{proof}
Let sequence $(u_n)$ of fuzzy numbers be Ces\`{a}ro summable to a fuzzy number $\mu$. Then sequence of Ces\`{a}ro means $\sigma_n=\frac{1}{n+1}\sum_{k=0}^nu_k$ converges to $\mu$. From Proposition \ref{pro} we have
\begin{eqnarray*}
D(u_n,\bar{0})=D\left(\sum_{k=0}^nu_k, \sum_{k=0}^{n-1}u_k\right)=D((n+1)\sigma_n, n\sigma_{n-1})\leq nD(\sigma_n,\sigma_{n-1})+D(\sigma_n,\bar{0})
\end{eqnarray*}
and, by dividing both sides with  $n$, we get
\begin{eqnarray*}
\frac{D(u_n,\bar{0})}{n}\leq D(\sigma_n,\sigma_{n-1})+\frac{D(\sigma_n,\bar{0})}{n}\cdot
\end{eqnarray*}
Since $(\sigma_n)$ is a convergent sequence, by limiting both sides we conclude $D(u_n,\bar{0})=o(n)$.

Now we shall show that the estimate $D(u_n,\bar{0})=o(n)$ is best possible. We prove by contradiction. Let estimate $D(u_n,\bar{0})=o\left(\frac{n}{\lambda_n}\right)$ be best possible for Ces\`{a}ro summable sequences $(u_n)$ of fuzzy numbers, where $(\lambda_{n})$ is a sequence of real numbers with $0<\lambda_n\neq O(1)$. Then there exists a subsequence $(\lambda_{n_k})$ of $(\lambda_{n})$ such that $n_{k+1}\geq n_{k}+2$ and $\lambda_{n_k}\uparrow \infty$. Then consider the sequence of fuzzy numbers $(u_n)$ defined by:
\begin{eqnarray*}
u_{n_k}(t)=
\begin{cases}
t-\frac{n_k}{\sqrt{\lambda_{n_k}}},  & \frac{n_k}{\sqrt{\lambda_{n_k}}}\leq t\leq \frac{n_k}{\sqrt{\lambda_{n_k}}}+1\\
2-t+\frac{n_k}{\sqrt{\lambda_{n_k}}}, &\frac{n_k}{\sqrt{\lambda_{n_k}}}+1\leq t\leq \frac{n_k}{\sqrt{\lambda_{n_k}}}+2\\
0,& otherwise,
\end{cases}\\
u_{n_k+1}(t)=\begin{cases}
t+\frac{n_k}{\sqrt{\lambda_{n_k}}},  & -\frac{n_k}{\sqrt{\lambda_{n_k}}}\leq t\leq 1-\frac{n_k}{\sqrt{\lambda_{n_k}}}\\
2-t-\frac{n_k}{\sqrt{\lambda_{n_k}}}, &1-\frac{n_k}{\sqrt{\lambda_{n_k}}}\leq t\leq 2-\frac{n_k}{\sqrt{\lambda_{n_k}}}\\
0,& otherwise,
\end{cases}
\end{eqnarray*}
for $n=n_k$, $n=n_{k}+1$  and
\begin{eqnarray*}
u_{n}(t)=
\begin{cases}
t, \quad & 0\leq t\leq 1\\
2-t, \quad & 1 \leq t\leq 2\\
0,& otherwise
\end{cases}
\end{eqnarray*}
for $n\neq \{n_k, n_k+1\}$. Then $\alpha-$level set of $(u_n)$ is
\begin{eqnarray*}
[u_{n_k}]_{\alpha} = \left[\alpha +\frac{n_k}{\sqrt{\lambda_{n_k}}}, 2-\alpha +\frac{n_k}{\sqrt{\lambda_{n_k}}}\right], \quad
[u_{n_k+1}]_{\alpha}=\left[\alpha -\frac{n_k}{\sqrt{\lambda_{n_k}}}, 2-\alpha -\frac{n_k}{\sqrt{\lambda_{n_k}}}\right]
\end{eqnarray*}
for $n=n_k$, $n=n_k+1$ and $[u_{n}]_{\alpha}=[\alpha, 2-\alpha]$ for $n\neq \{n_k, n_k+1\}$.
So $\alpha-$level set of Ces\`{a}ro means $(\sigma_n)$ is
\begin{eqnarray*}
[\sigma_{n_k}]_{\alpha}=\left[\alpha +\frac{n_k}{(n_k+1)\sqrt{\lambda_{n_k}}}, 2-\alpha +\frac{n_k}{(n_k+1)\sqrt{\lambda_{n_k}}}\right]
\end{eqnarray*}%
for $n=n_k$ and $[\sigma_{n}]_{\alpha}=\left[\alpha, 2-\alpha\right]$ for $n\neq n_k$. Thus we conclude that sequence $(u_n)$ is Ces\`{a}ro summable to fuzzy number
\begin{eqnarray*}
\mu(t)=
\begin{cases}
t, \quad & 0\leq t\leq 1\\
2-t, \quad & 1 \leq t\leq 2\\
0,& otherwise.
\end{cases}
\end{eqnarray*}
However we have
\begin{eqnarray*}
\frac{\lambda_{n_k}}{n_k}D(u_{n_k},\bar{0})=\frac{2\lambda_{n_k}}{n_k}+\sqrt{\lambda_{n_k}}\to\infty
\end{eqnarray*}
as $k\to\infty$, which contradicts with the assumption $D(u_n,\bar{0})=o\left(\frac{n}{\lambda_n}\right)$. The proof is completed.
\end{proof}
Now we prove a theorem dealing with multiplication of infinite series of fuzzy numbers, which is analogous to Mertens' theorem that in classical analysis.
\begin{thm}\label{mertens}
Let $\sum_{n=0}^{\infty} u_n$ be a convergent series of fuzzy numbers. If $\sum_{n=0}^{\infty} x_n$ is a convergent series with non-negative real terms, then
\begin{eqnarray*}
\left(\sum_{n=0}^{\infty} x_n\right)\left(\sum_{n=0}^{\infty} u_n\right)=\sum_{n=0}^{\infty}\sum_{k=0}^nu_kx_{n-k}.
\end{eqnarray*}
\end{thm}
\begin{proof}
Let $\sum_{n=0}^{\infty} u_n$ be a convergent series of fuzzy numbers and $\sum_{n=0}^{\infty} x_n$ be a convergent series with non-negative real terms. Then there exist $U\in E^1$ and $X\in \mathbb{R}$ such that $U_n=\sum_{k=0}^{n} u_k\rightarrow U$ and $X_n=\sum_{k=0}^{n} x_k\rightarrow X$ are satisfied.  Hence for given any $\varepsilon>0$
\begin{itemize}
\item[(i)] there exists $n_0\in\mathbb{N}$ such that $D(U_n, U)<\frac{\varepsilon}{3(X+1)}$ whenever $n>n_0$,\\[0.3mm]
\item[(ii)] there exists  $n_1\in\mathbb{N}$ such that for  $n>n_1$ we have $x_n<\frac{\varepsilon}{3\left\{(n_0+1)\max\limits_{0\leq k\leq n_0}\{D(U_k,U)\}+1\right\}}$\\[0.3mm]
\item[(iii)] there exists $n_2\in\mathbb{N}$ such that for $n>n_2$ we have $\sum_{k=n+1}^{\infty}x_k<\frac{\varepsilon}{3(D(U,\bar{0})+1)}\cdot$
\end{itemize}
On the other hand by Remark \ref{change} we have
\begin{eqnarray*}
\sum_{n=0}^{m}\sum_{k=0}^{n}u_kx_{n-k}=\sum_{n=0}^{m}\sum_{k=0}^{n}x_ku_{n-k}=\sum_{k=0}^{m}x_k\sum_{n=k}^{m}u_{n-k}=\sum_{k=0}^{m}x_k\sum_{n=0}^{m-k}u_{n}=\sum_{k=0}^{m}x_kU_{m-k}.
\end{eqnarray*}
Since
\begin{eqnarray*}
D\bigg(\sum_{n=0}^{m}\sum_{k=0}^{n}u_kx_{n-k}, XU\bigg)&=&D\left(\sum_{k=0}^{m}x_kU_{m-k}, \sum_{k=0}^{\infty}x_kU\right)
\\&=&
D\left(\sum_{k=0}^{m}x_kU_{m-k}, \sum_{k=0}^{m}x_kU+\sum_{k=m+1}^{\infty}x_kU\right)
\\&\leq&
D\left(\sum_{k=0}^{m}x_kU_{m-k}, \sum_{k=0}^{m}x_kU\right)+D\left(\sum_{k=m+1}^{\infty}x_kU, \bar{0}\right)
\\&\leq&
\sum_{k=0}^{m-n_0-1}x_kD(U_{m-k},U)+\sum_{k=m-n_0}^{m}x_kD(U_{m-k},U)
+D\left(\sum_{k=m+1}^{\infty}x_kU, \bar{0}\right),
\end{eqnarray*}
we get
\begin{eqnarray*}
D\left(\sum_{n=0}^{m}\sum_{k=0}^{n}u_kx_{n-k}, XU\right)< \varepsilon
\end{eqnarray*}
whenever $m>\max\{n_0+n_1,n_2\}$, and this completes the proof.
\end{proof}
\begin{thm}\label{cesaroabel}
If sequence $(u_n)$ of fuzzy numbers is Ces\`{a}ro summable to fuzzy number $\mu$, then  $(u_n)$ is Abel summable to $\mu$.
\end{thm}
\begin{proof}
Let $(u_n)$ be Ces\`{a}ro summable to a fuzzy number $\mu$. We want to show that series $\sum u_nx^n$ of fuzzy numbers is convergent for $x\in(0,1)$, and
\begin{eqnarray*}
\lim_{x\to1^-}(1-x)\sum_{n=0}^{\infty}u_nx^n=\mu.
\end{eqnarray*}
From Theorem \ref{bound} we have $D(u_n,\bar{0})=o(n)$ and as result we get
\begin{eqnarray*}
\sum_{n=0}^{\infty}D(u_nx^n,\bar{0})&\leq&\sum_{n=0}^{\infty}D(u_n,\bar{0})x^n\leq \sum_{n=0}^{\infty}nx^n=\frac{x}{(1-x)^2}
\end{eqnarray*}
where $x\in(0,1)$. So by Theorem \ref{absolute}, series $\sum u_nx^n$ of fuzzy numbers is convergent for $x\in(0,1)$. Besides, from Theorem \ref{mertens} we get
\begin{eqnarray*}
(1-x)\sum_{n=0}^{\infty}u_nx^n&=&(1-x)^2\frac{1}{1-x}\sum_{n=0}^{\infty}u_nx^n=(1-x)^2\left(\sum_{n=0}^{\infty}x^n\right)\left(\sum_{n=0}^{\infty}u_nx^n\right)
\\&=&
(1-x)^2\sum_{n=0}^{\infty}s_nx^n=(1-x)^2\sum_{n=0}^{\infty}(n+1)\sigma_nx^n.
\end{eqnarray*}
At this point we recall the power series method $(J,p)$ introduced by Sefa and \c{C}anak\cite{sefapower}. Since sequence $(\sigma_n)$ of Ces\`{a}ro means converges to  $\mu$ and summability method $(J,n+1)$ is regular we have
\begin{eqnarray*}
\lim_{x\to1^-}(1-x)^2\sum_{n=0}^{\infty}(n+1)\sigma_nx^n=\mu,
\end{eqnarray*}
from which we conclude
\begin{eqnarray*}
\lim_{x\to1^-}(1-x)\sum_{n=0}^{\infty}u_nx^n=\mu.
\end{eqnarray*}
\end{proof}
However Abel summable sequences of fuzzy number do not have to be Ces\`{a}ro summable, which can be seen by following example.
\begin{ex}
Consider sequence $u=(u_n)$ of fuzzy numbers such that

\begin{eqnarray*}
u_n(t)=
\begin{cases}
2\sqrt[n]{t+(-1)^{n+1}n}, &  if \  (-1)^nn\leq t\leq (-1)^nn+\frac{1}{2^n}, \\[0.1 cm]
1 , & if \  (-1)^nn+\frac{1}{2^n}\leq t\leq (-1)^nn+2-\frac{1}{2^n}, \\[0.1 cm]
2\sqrt[n]{(-1)^{n}n+2-t} , & if \ (-1)^nn+2-\frac{1}{2^n}\leq t\leq (-1)^nn+2, \\[0.1 cm]
0, & \qquad  \quad otherwise
\end{cases}
\end{eqnarray*}
for $n\geq 1$ and $u_0=\bar{1}$. Since
\begin{eqnarray*}
\sum\limits_{n=0}^{\infty} u^-_n(\alpha)x^n&=&\sum\limits_{n=0}^{\infty}\left\{(-1)^nn+\left(\frac{\alpha}{2}\right)^n\right\}x^n=\frac{-x}{(1+x)^2}+\frac{2}{2-\alpha x}
\\
\sum\limits_{n=0}^{\infty} u^+_n(\alpha)x^n&=&\sum\limits_{n=0}^{\infty}\left\{(-1)^nn+2-\left(\frac{\alpha}{2}\right)^n\right\}x^n= \frac{-x}{(1+x)^2}+\frac{2}{1-x}-\frac{2}{2-\alpha x}
\end{eqnarray*}
converges uniformly in $\alpha$ where $0<x<1$, series $\sum u_nx^n$ is convergent by Definition \ref{series}. Then considering the fuzzy number $\mu$, where $[\mu]_\alpha=[0,2]$, we get
\begin{eqnarray*}
D\left((1-x)\sum u_nx^n, \mu\right)&=&
\sup_{\alpha\in[0,1]}\max\left\{\left|\frac{-(1-x)x}{(1+x)^2}+\frac{2(1-x)}{2-\alpha x}\right|,
\left|\frac{-(1-x)x}{(1+x)^2}+\frac{2(1-x)}{1-x}-\frac{2(1-x)}{2-\alpha x}-2\right|\right\}
\\&=&
\sup_{\alpha\in[0,1]}\left|\frac{-(1-x)x}{(1+x)^2}-\frac{2(1-x)}{2-\alpha x}\right|=\frac{(1-x)x}{(1+x)^2}+\frac{2(1-x)}{2-\alpha x}
\end{eqnarray*}
and so $\lim_{x\to 1^-}(1-x)\sum u_nx^n=\mu$.
Hence sequence $(u_n)$ of fuzzy numbers is Abel summable to fuzzy number
\begin{eqnarray*}
\mu(t)=
\begin{cases}
1, \quad & 0\leq t\leq 2\\
0,& otherwise,
\end{cases}
\end{eqnarray*}
but is not  Ces\`{a}ro summable to any fuzzy number.
\end{ex}
\section{Comparison between Euler and Borel methods of summability of sequences of fuzzy numbers}
\begin{thm}\label{transform}
Let $(u_n)$ be a sequence of fuzzy numbers. Then $q$-th order Euler means of $p$-th order Euler means of $(u_n)$ is $(p+q+pq)$-th order Euler means of $(u_n)$.
\end{thm}
\begin{proof}
Let $(u_n)$ be a sequence of fuzzy numbers and $(t^{p}_n)$ be the sequence of $p$-th order Euler means of $(u_n)$. Then sequence of $q$-th order Euler means of $(t^{p}_n)$ is
\begin{eqnarray*}
t^{q}_{n}(t_n^p)&=&\frac{1}{(q+1)^n}\sum_{k=0}^{n}\binom{n}{k}q^{n-k}t_n^p
\\&=&
\frac{1}{(q+1)^n}\sum_{k=0}^{n}\binom{n}{k}q^{n-k}\frac{1}{(p+1)^k}\sum_{m=0}^{k}\binom{k}{m}p^{k-m}u_m
\\&=&
\frac{1}{(q+1)^n}\sum_{m=0}^{n}u_m\sum_{k=m}^{n}\binom{n}{k}\binom{k}{m}q^{n-k}p^{k-m}\frac{1}{(p+1)^k}
\\&=&
\frac{1}{(q+1)^n}\sum_{m=0}^{n}\binom{n}{m}\frac{u_m}{(p+1)^m}\sum_{k=0}^{n-m}\binom{n-m}{k}\left(\frac{p}{p+1}\right)^{k}q^{n-m-k}
\\&=&
\frac{1}{(q+1)^n}\sum_{m=0}^{n}\binom{n}{m}\left(\frac{1}{p+1}\right)^m\left(\frac{pq+p+q}{p+1}\right)^{n-m}u_m
\\&=&
\frac{1}{(pq+p+q+1)^n}\sum_{m=0}^{n}\binom{n}{m}(pq+p+q)^{n-m}u_m
\\&=&
t_n^{pq+p+q}
\end{eqnarray*}
in view of Remark \ref{change}, which completes the proof.
\end{proof}
\begin{thm}
If sequence $(u_n)$ of fuzzy numbers is $E_p$ summable to a fuzzy number $\mu$ , and $s>p>0$, then it is $E_{s}$ summable to $\mu$.
\end{thm}
\begin{proof}
Let $s>p>0$ and let sequence $(u_n)$ of fuzzy numbers be $E_p$ summable to a fuzzy number $\mu$. Then sequence $(t^{p}_n)$ of Euler means of $(u_n)$ converges to $\mu$. Besides it follows from Theorem \ref{transform} that $t_n^{s}=t^{\frac{s-p}{p+1}}_{n}(t_n^p)$. By regularity of Euler summability method we conclude that $(t^{s}_n)\to \mu$ and this completes the proof.
\end{proof}
But $E_{s}$ summable sequences are not necessarily $E_{p}$ summable for  $s>p>0$, which can be seen by following example.
\begin{ex}
Let $(u_n)$  be a sequence of fuzzy number  such that
\begin{eqnarray*}
u_n(t)=
\begin{cases}
2\ \sqrt[n]{t-(-p-s-1)^n}, &(-p-s-1)^n\leq t\leq (-p-s-1)^n+\frac{1}{2^n}\\
1, &(-p-s-1)^n+\frac{1}{2^n}\leq t\leq(-p-s-1)^n+1\\
(-p-s-1)^n+2-t, &(-p-s-1)^n+1\leq t\leq (-p-s-1)^n+2\\
0,& (otherwise)
\end{cases}
\end{eqnarray*}
for $n\geq1$ and $[u_0]_{\alpha}=[2,3-\alpha]$. Then
\begin{eqnarray*}
[u_n]_{\alpha}=\left[(-p-s-1)^n+\left(\frac{\alpha}{2}\right)^n, (-p-s-1)^n+2-\alpha \right].
\end{eqnarray*}
So $\alpha-$level set of sequence $t^s_n$ of $s$-th order Euler means is
\begin{eqnarray*}
\left[t^s_n\right]_{\alpha}&=&\left[\frac{1}{(s+1)^n}\sum_{k=0}^n\binom{n}{k}s^{n-k}\left\{(-p-s-1)^k+\left(\frac{\alpha}{2}\right)^k\right\},
\frac{1}{(s+1)^n}\sum_{k=0}^n\binom{n}{k}s^{n-k}\left\{(-p-s-1)^k+2-\alpha\right\}\right]\\
&=&\left[(-1)^n\frac{(p+1)^n}{(s+1)^n}+\left(\frac{2s+\alpha}{2s+2}\right)^n, (-1)^n\frac{(p+1)^n}{(s+1)^n}+2-\alpha\right].
\end{eqnarray*}
Hence $D\left(t^s_n,\mu\right)=\frac{(p+1)^n}{(s+1)^n}+\left(\frac{2s+1}{2s+2}\right)^n\to 0$ where
$$\mu(t)=
\begin{cases}
t, \quad &0\leq t\leq 1\\
2-t, \quad &1\leq t\leq 2\\
0,& (otherwise).
\end{cases}$$
So we conclude that sequence $(u_n)$ is $E_s$ summable to fuzzy number $\mu$. Now let investigate the $E_p$ summabiltiy of $(u_n)$. $\alpha-$level set of sequence $t^p_n$ of $p$-th order Euler means is
\begin{eqnarray*}
\left[t^p_n\right]_{\alpha}&=&\left[\frac{1}{(p+1)^n}\sum_{k=0}^n\binom{n}{k}p^{n-k}\left\{(-p-s-1)^k+\left(\frac{\alpha}{2}\right)^k\right\},\frac{1}{(p+1)^n}\sum_{k=0}^n\binom{n}{k}p^{n-k}\left\{(-p-s-1)^k+2-\alpha\right\}\right]
\\&=&
\left[(-1)^n\frac{(s+1)^n}{(p+1)^n}+\left(\frac{2p+\alpha}{2p+2}\right)^n, (-1)^n\frac{(s+1)^n}{(p+1)^n}+2-\alpha\right]
\end{eqnarray*}%
and then sequence $(u_n)$ is not $E_p$ summable to any number $\mu$ since sequence $\left[t^p_n\right]_{\alpha}$ is not convergent.
\end{ex}
Now we prove a lemma which is necessary to achieve the goal of this section.
\begin{lem}\label{lazým}
Let $\sum_{n=0}^{\infty} u_n$ be a convergent series of fuzzy numbers. If $\sum_{n=0}^{\infty} x_n$ is a convergent series with non-negative real terms, then
\begin{eqnarray*}
\lim_{n\to\infty}\sum_{k=0}^nu_k\sum_{v=k}^nx_v=\sum_{k=0}^{\infty}u_k\sum_{v=k}^{\infty}x_v.
\end{eqnarray*}
\end{lem}
\begin{proof}
Let $\sum_{n=0}^{\infty} u_n$ be a convergent series of fuzzy numbers and $\sum_{n=0}^{\infty} x_n$ be a convergent series with non-negative real terms. Then we have
\begin{eqnarray*}
D\left(\sum_{k=0}^nu_k\sum_{v=k}^nx_v, \sum_{k=0}^{\infty}u_k\sum_{v=k}^{\infty}x_v\right)&=&
D\left(\sum_{k=0}^nu_k\sum_{v=k}^nx_v,\sum_{k=0}^nu_k\sum_{v=k}^nx_v+ \sum_{k=0}^nu_k\sum_{v=n+1}^{\infty}x_v+\sum_{k=n+1}^{\infty}u_k\sum_{v=k}^{\infty}x_v\right)
\\&\leq&
D\left(\sum_{k=0}^nu_k\sum_{v=n+1}^{\infty}x_v, \bar{0}\right)+D\left(\sum_{k=n+1}^{\infty}u_k\sum_{v=k}^{\infty}x_v, \bar{0}\right)
\\&\leq&
\left\{\sum_{v=n+1}^{\infty}x_v\right\}D\left(\sum_{k=0}^nu_k, \bar{0}\right)+\left\{\sum_{v=0}^{\infty}x_v\right\}D\left(\sum_{k=n+1}^{\infty}u_k, \bar{0}
\right).
\end{eqnarray*}
Since series $\sum_{n=0}^{\infty} u_n$ and $\sum_{n=0}^{\infty} x_n$ are convergent, both of series are bounded and corresponding remainder terms converge to $0$ as $n\to\infty$. So by limiting both sides of the expression above we get
\begin{eqnarray*}
\lim_{n\to\infty} D\left(\sum\limits_{k=0}^nu_k\sum\limits_{v=k}^nx_v, \sum\limits_{k=0}^{\infty}u_k\sum\limits_{v=k}^{\infty}x_v\right)=0,
\end{eqnarray*}
and the proof is completed.
\end{proof}
\begin{thm}\label{eulerborel}
If sequence $(u_n)$ of fuzzy numbers is $E_p$ summable to a fuzzy number $\mu$ , then it is Borel summable to $\mu$.
\end{thm}
\begin{proof}
Let sequence $(u_n)$ of fuzzy numbers be $E_p$ summable to fuzzy number $\mu$. Our aim is to show that  $\sum_{n=0}^{\infty} \frac{x^n}{n!}u_n$
converges for all $x\in(0,\infty)$ and
\begin{eqnarray*}
\lim_{x\to \infty}e^{-x}\sum\limits\limits_{n=0}^{\infty} \frac{x^n}{n!}u_n = \mu.
\end{eqnarray*}
Since sequence $(u_n)$ is $E_p$ summable to $\mu$, sequence a of Euler means converges to $\mu$.Then we have $\left[t^p_n\right]_{\alpha}\to[\mu]_{\alpha}$ for $0\leq \alpha \leq 1$, which, in special case, implies sequences $(u^-_n(0))$ and $(u^+_n(0))$ are  $E^p$ summable to ${\mu}^-_n(0)$ and ${\mu}^+_n(0)$, respectively. Then we have $u^-_n(0)=o((2p+1)^n)$ and $u^+_n(0)=o((2p+1)^n)$. So we get
\begin{eqnarray*}
D(u_n,\overline{0})=\max\{|u^-(0)|,|u^+(0)|\}=o((2p+1)^n).
\end{eqnarray*}
By using this fact, for all $x\in(0,\infty)$ we have
\begin{eqnarray*}
\sum\limits_{n=0}^{\infty} D\left(\frac{x^n}{n!}u_n,\bar{0}\right)&\leq& \sum\limits_{k=0}^{\infty} D\left(u_n,\bar{0}\right)\frac{x^n}{n!}\leq
\sum\limits_{n=0}^{\infty}\frac{((2p+1)x)^n}{n!}=e^{(2p+1)x},
\end{eqnarray*}
and so from Thereom \ref{absolute} series $\sum\limits_{n=0}^{\infty} \frac{x^n}{n!}u_n$ converges for all $x\in(0,\infty)$. Besides we have
\begin{eqnarray*}
\sum_{n=0}^{m}(p+1)^nt^p_n\frac{x^n}{n!}=\sum_{n=0}^{m}\frac{x^n}{n!}\sum_{k=0}^{n}\binom{n}{k}p^{n-k}u_k=\sum_{k=0}^{m}u_k\sum_{n=k}^{m}\binom{n}{k}\frac{x^n}{n!}p^{n-k}
=\sum_{k=0}^{m}\frac{x^k}{k!}u_k\sum_{n=k}^{m}\frac{(px)^{n-k}}{(n-k)!}
\end{eqnarray*}
and by Lemma \ref{lazým} we get
\begin{eqnarray*}
\sum_{n=0}^{\infty}t^p_n\frac{[(p+1)x]^n}{n!}&=&e^{px}\sum_{k=0}^{\infty}\frac{x^k}{k!}u_k.
\end{eqnarray*}
Dividing both sides by $e^{(p+1)x}$ it follows that
\begin{eqnarray*}
\frac{1}{e^{(p+1)x}}\sum_{n=0}^{\infty}t^p_n\frac{[(p+1)x]^n}{n!}&=&e^{-x}\sum_{k=0}^{\infty}\frac{x^k}{k!}u_k.
\end{eqnarray*}
Finally, since $(t^p_n)\to \mu$ and Borel summability method is regular, by limiting both sides as $x\to\infty$ we conclude that
\begin{eqnarray*}
\lim_{x\to \infty}e^{-x}\sum\limits\limits_{n=0}^{\infty} \frac{x^n}{n!}u_n = \mu.
\end{eqnarray*}
\end{proof}
Borel summability of a sequence of fuzzy numbers may not imply $E_p$ summability. This can be seen by sequence $(u_n)$  of fuzzy numbers defined by
\begin{eqnarray*}
u_n(t)=
\begin{cases}
t-(-1)^nn!, \quad &(-1)^nn!\leq t\leq (-1)^nn!+1\\[5pt]
\dfrac{(-1)^nn!+2-t}{t-(-1)^nn!}, &(-1)^nn!+1\leq t\leq(-1)^nn!+2\\[5pt]
0,& (otherwise).
\end{cases}
\end{eqnarray*}
Sequence $(u_n)$ of fuzzy numbers is Borel summable to fuzzy number
$$\mu(t)=
\begin{cases}
t, \quad &0\leq t\leq 1\\[3pt]
\dfrac{2-t}{t}, \quad &1\leq t\leq 2\\[3pt]
0,& (otherwise),
\end{cases}$$
but not $E_p$ summable to any fuzzy number.
\section{Conclusion}
In this study we have proved comparison theorems for recently introduced summability methods of sequences of fuzzy numbers. Besides, various results dealing with series of fuzzy numbers have been obtained. A comparison theorem, in general, provides us with the facility of extending the results of one method to another one directly without needing a separate proof. So it makes possible to utilize from the results in one method to achive the goals related with the other method. In our case, in view of Theorem \ref{cesaroabel} and Theorem \ref{eulerborel}, we can extend the results for Abel summability method of sequences of fuzzy numbers\cite{yavuzabel} and Borel summability method of sequences of fuzzy numbers\cite{yavuzborel} to Ces\`{a}ro and Euler summability methods, respectively. We mention some of these results concerning the convergence of summable sequences of fuzzy numbers below.
\begin{cor}
If sequence $(u_n)$ of fuzzy numbers is Ces\`{a}ro summable to fuzzy number $\mu$ and $nD(u_n,u_{n-1})=o(1)$, then sequence $(u_n)$ converges to $\mu$.
\end{cor}
\begin{cor}
If series $\sum u_n$ of fuzzy numbers is Ces\`{a}ro summable to fuzzy number $\nu$ and $nD(u_n,\bar{0})=o(1)$, then $\sum u_n=\nu$.
\end{cor}
\begin{cor}\cite{yavuzeuler}
If sequence $(u_n)$ of fuzzy numbers is $E_p$ summable to fuzzy number $\mu$ and  $\sqrt{n} D(u_{n-1}, u_{n})=o(1)$, then $(u_n)$ converges to $\mu$.
\end{cor}
\begin{cor}\cite{yavuzeuler}
If series $\sum u_n$ of fuzzy numbers is $E_p$ summable to fuzzy number $\nu$ and  $\sqrt{n}D(u_n, \bar{0})=o(1)$, then $\sum u_n=\nu$.
\end{cor}

\end{document}